\documentclass
{amsart}
\usepackage{graphicx}
\newtheorem{thm}{Theorem}[section]
\theoremstyle{definition}
\newtheorem{cor}[thm]{Corollary}
\newtheorem{prop}[thm]{Proposition}
\newtheorem{defn}[thm]{Definition}

\newtheorem{rem}[thm]{Remark}
\newtheorem{ex}[thm]{Example}

\numberwithin{equation}{section}
\begin{document}
\title[$n$-pure submodules of modules]
{$n$-pure submodules of modules}
\author{F. Farshadifar}

\address{Assistant Professor, Department of Mathematics,  Farhangian University, Tehran, Iran.}
                          \email{f.farshadifar@cfu.ac.ir}

\begin{abstract}
Let $R$ be a commutative ring, $M$ an $R$-module, and $n\geq 1$ an integer. In this paper, we will introduce the concept of  $n$-pure submodules of $M$ as a generalization of pure submodules and obtain some related results.
\end{abstract}
\maketitle
\section{Introduction}
\noindent
Throughout this paper, $R$ will denote a commutative ring with
identity and $\Bbb Z$ will denote the ring of integers. Further,  $n$ will denote a positive integer.

Let $M$ be an $R$-module.
$M$ is said to be a \emph{multiplication module} if for every submodule $N$ of $M$, there exists an ideal $I$ of $R$ such that $N=IM$ \cite{Ba81}. It is easy to see that $M$ is a multiplication module if and only if $N=(N:_RM)M$ for each submodule $N$ of $M$.

Cohn \cite{Co59} defined a submodule $N$ of $M$ a \textit{pure submodule}
if the sequence $0 \rightarrow N\otimes E \rightarrow M  \otimes E$ is exact for every $R$-module $E$. Anderson and Fuller
\cite{AF74} called the submodule $N$ a \textit{ pure submodule} of $M$ if $IN=N \cap IM$ for every ideal $I$ of
$R$. Ribenboim \cite{Ri72} called  $N$ to be \textit{pure} in $M$ if $rM\cap N = rN$ for each $r\in R$. Although
the first condition implies the second \cite[p.158]{La99},  and the second obviously implies the third,
these definitions are not equivalent in general, see \cite[p.158]{La99} for an example.
The three definitions of purity given above are equivalent if $M$ is flat.  In particular, if $M$ is a faithful multiplication
module \cite{AS04}.

In this paper,
our definition of purity will be that of Anderson and Fuller
\cite{AF74}.

Let $n$ be a positive integer.
The main purpose of this paper is to introduce the concept of $n$-pure submodules of an $R$-module $M$ as a generalization of pure submodules and investigate some results concerning this notion.

\section{2-pure submodules}

\begin{defn}\label{d5.1}
We say that a submodule $N$ of an $R$-module $M$ is a \emph {2-pure submodule of $M$} if $IJN=IN \cap JN\cap IJM$ for all proper ideals $I, J$ of $R$. Also, we say that an ideal $I$ of $R$ is a \emph {2-pure ideal of $R$} if $I$ is a 2-pure submodule of $R$.
\end{defn}

\begin{rem}\label{r5.2}
Clearly every pure submodule of an $R$-module $M$ is a 2-pure submodule of $M$. But we see in the Example \ref{e5.3} that the converse is not true in general.
\end{rem}

\begin{ex}\label{e5.3}
The submodule $\bar{2}\Bbb Z_4$ of the $\Bbb Z_4$-module $\Bbb Z_4$ is a 2-pure submodule but it is not a  pure submodule.
\end{ex}

\begin{ex}\label{e55.3}
\begin{itemize}
  \item [(a)] The submodule $2\Bbb Z$ of the $\Bbb Z$-module $\Bbb Z$ is a not a 2-pure submodule.
 \item [(b)] Take $R=K[x,y]$, where $K$ is a field, and take $M=R^2/(xy,y^2)(x+y)R$. Then one can see that the submodule $N=\bar{(x,y)}R$ is not a 2-pure submodule of $M$.
 \item [(c)] Since $4\Bbb Z=(2\Bbb Z)(2\Bbb Z) \Bbb Z\not = (2\Bbb Z)\Bbb Z \cap (2\Bbb Z) \Bbb Z \cap (4\Bbb Z) \Bbb Q=2\Bbb Z$,
      the submodule $\Bbb Z$ of the $\Bbb Z$-module $\Bbb Q$ is not 2-pure.
\end{itemize}
\end{ex}

A non-zero submodule $N$ of an $R$-module $M$  is said to be  a \emph{weakly strongly 2-absorbing second submodule of} $M$ if whenever
 $a, b \in R$, $K$ is a submodule of $M$, $abM \not \subseteq K$, and $abN\subseteq K$, then $aN\subseteq K$ or
$bN\subseteq K$ or $ab \in Ann_R(N)$ \cite{AF105}.

\begin{prop}\label{l15.3}
Let $N$ be a submodule of an $R$-module $M$. Then we have the following.
\begin{itemize}
  \item [(a)] If $IJN=IN \cap JN$ for all proper ideals $I, J$ of $R$, then $N$ is a 2-pure submodule of $M$.
  \item [(b)] If for each ideal $I$ of $R$, $IN$ is a pure submodule of $M$,  then $N$ is a 2-pure submodule of $M$.
  \item [(c)] If $N$ is a weakly strongly $2$-absorbing second submodule of $M$, then for any $ab \in  R\setminus Ann_R(N)$, we have $abN=aN \cap bN \cap abM$.
\end{itemize}
\end{prop}
\begin{proof}
(a) Let $I$ and $J$ be proper ideals of $R$. Then by assumption, $IJN=IN \cap JN$. Thus $IN \cap JN=IJN\subseteq IJM$. This implies that $IN \cap JN\cap IJM=IN \cap JN$. Therefore,  $IN \cap JN\cap IJM=IJN$ as required.

(b) Let $I$ and $J$ be proper ideals of $R$. Then by assumption, $I(JN)=JN \cap IM \supseteq JN \cap IN$. This implies that $N$ is a 2-pure submodule of $M$ since the reverse inclusion is clear.

(c) Let $N$ be a weakly strongly $2$-absorbing second submodule of $M$ and $ab \in  R\setminus Ann_R(N)$. Clearly, $abN\subseteq aN \cap bN \cap abM$. Now let $L$ be a completely irreducible submodule of $M$ such that $abN \subseteq L$. If $abM  \subseteq L$, then we are done. If $abM \not \subseteq L$, then $aN \subseteq L$ or $bN \subseteq L$ because $N$ is a weakly strongly $2$-absorbing second submodule of $M$ and the proof is completed.
\end{proof}

\begin{prop}\label{p5.9}
Let $M$ be an $R$-module and $N\subseteq K$ be two submodules of $M$. Then the 2-absorbing purity relation
satisfies the following.
\begin{itemize}
  \item [(a)] Transitive: If $N$ is a 2-pure submodule of $K$ and $K$ is a 2-pure submodule of $M$, then $N$ is a 2-pure submodule of $M$.
  \item [(b)] Hereditary: If  $K$ is a 2-pure submodule of $M$, then $K/N$ is a 2-pure submodule of $M/N$.
  \item [(c)] If $K/N$ is a 2-pure submodule of $M/N$ and $N$ is a pure submodule of $M$, then $K$ is a 2-pure submodule of $M$.
\end{itemize}
\end{prop}
\begin{proof}
(a) and (b) are clear.

(c) Let $I$ and $J$ be two proper ideals of $R$. Since $K/N$ is a 2-pure submodule of $M/N$, we have
 $$
 IJK+N=(IK+N)\cap (JK+N)\cap (IJM+N)\supseteq (IK\cap JK\cap IJM)+N .
 $$
 Now let $x \in IK\cap JK\cap IJM$. Then $x+N \in IJK+N$. Let $n_1 \in N$. Then $x+n_1=y+n_2$ for some $y \in IJK$ and $n_2 \in N$. This implies that $x -y \in N$. Therefore, $x -y \in N\cap IJM=IJN$. It follows that $x-y \in IJK$ and so $x \in IJK$. Thus $JK \cap IK \cap IJM \subseteq IJK$. This implies that $K$ is a 2-pure submodule of $M$ since the reverse inclusion is clear.
\end{proof}

\begin{prop}\label{l5.4}
Let $R$ be a $PID$, $N$ a submodule of an $R$-module $M$, and $p_i$  ($i \in \Bbb N$) be a prime element in $R$. Then $p_1^{s_1}...p_t^{s_t}N=\cap^t_{i=1}p_i^{s_i}N$.
\end{prop}
\begin{proof}
Let $p$ and $q$ be two prime elements in $R$ and $k, s \in \Bbb N$. Clearly, $p^kq^sN \subseteq p^tN \cap q^sN$. Now let $x \in p^kN \cap q^sN$. Then $x=p^kn_1=q^sn_2$ for some $n_1,n_2 \in N$. Since $R$ is a $PID$, $p^kR+q^sR=R$. Thus there exist $a , b \in R$ such that $1=ap^k+bq^s$. Hence $x=1x=ap^kx+bq^sx=ap^kq^sn_2+bq^sp^kn_1$. Thus $x \in p^kq^sN$. Therefore, $p^kq^sN=p^tN \cap q^sN$. Now the result follows by induction on $t$.
\end{proof}

\begin{cor}\label{p5.6}
Let $M$ be a $\Bbb Z$-module, $m,n$ be square-free integers such that $(m,n)=1$. Then for each submodule $N$ of $M$ we have $(n\Bbb Z)( m\Bbb Z) N=(n\Bbb Z) N \cap (m\Bbb Z) N$.
\end{cor}
\begin{proof}
This follows from Proposition \ref{l5.4}.
\end{proof}

\begin{prop}\label{p5.6}
A submodule $N $ of an $R$-module $M $ is a  2-pure submodule if and only if  $N_P$ is a  2-pure submodule of $M_P$ for every maximal ideal $P$ of $R$.
\end{prop}
\begin{proof}
This follows from the fact that for each ideal $I$ and $J$ of $R$, by \cite[9.11]{sh90},
$$
I_PJ_PN_P=(IJN)_P=(IN \cap JN \cap IJM)_P=I_PN_P \cap J_PN_P \cap I_PJ_PM_P.
$$
\end{proof}

\begin{prop}\label{p5.8}
Let $M$ be an $R$-module. Then we have the following.
\begin{itemize}
  \item [(a)] If $\{N_\lambda\}_{\lambda \in \Lambda}$ is a chain of 2-pure submodules of $M$, then
  $\cup_{\lambda \in \Lambda}N_\lambda$ is  a 2-pure submodules of $M$.
  \item [(b)] If $\{N_\lambda\}_{\lambda \in \Lambda}$ is a chain of submodules of $M$ and $K$ is a 2-pure submodule of $N_\lambda$ for each $\lambda \in \Lambda$, then $K$ is a 2-pure submodule of $\cup_{\lambda \in \Lambda}N_\lambda$.
\end{itemize}
\end{prop}
\begin{proof}
(a) Let $I$ and $J$ be two proper ideals of $R$. Clearly,
$$
IJ(\cup_{\lambda \in \Lambda}N_\lambda)\subseteq I(\cup_{\lambda \in \Lambda}N_\lambda) \cap J(\cup_{\lambda \in \Lambda}N_\lambda)\cap IJM.
$$
Now let $x \in I(\cup_{\lambda \in \Lambda}N_\lambda) \cap J(\cup_{\lambda \in \Lambda}N_\lambda)\cap IJM$. Then $x=\sum^n_{i=1}a_in_i=\sum^m_{j=1}b_jm_j$, where $a_i\in I$, $b_j \in I$ and $n_i,m_j \in \cup_{\lambda \in \Lambda}N_\lambda$. Since $\{N_\lambda\}_{\lambda \in \Lambda}$ is a chain, there exists $\alpha \in \Lambda$ such that $x \in IN_{\alpha}$ and $x \in JN_{\alpha}$. Thus $x \in IN_{\alpha} \cap JN_{\alpha} \cap IJM$. This implies that $x \in IJN_{\alpha} \subseteq IJ(\cup_{\lambda \in \Lambda}N_{\lambda})$ as needed.

(b) Let $I$ and $J$ be two proper ideals of $R$. Clearly,
$$
IJK\subseteq IK \cap JK\cap IJ(\cup_{\lambda \in \Lambda}N_\lambda).
$$
Now let $x \in IK \cap JK\cap IJ(\cup_{\lambda \in \Lambda}N_\lambda)$. Then $x=\sum^n_{i=1}a_ib_in_i$, where $a_i \in I$, $b_j \in J$ and $n_i\in \cup_{\lambda \in \Lambda}N_\lambda$. Since $\{N_\lambda\}_{\lambda \in \Lambda}$ is a chain, there exists $\beta \in \Lambda$ such that $x \in IJN_{\beta}$. Thus  $x \in IK \cap JK\cap IJN_{\beta}$. This in turn implies that $x \in IJK$ as required.
\end{proof}

\begin{thm}\label{p5.9}
Let $N$ be a submodule of an $R$-module $M$. Then there is a submodule $K$ of $N$ maximal with respect to $K \subseteq N$ and $K$ is a 2-pure submodule of $M$.
\end{thm}
\begin{proof}
Let
$$
\Sigma= \{ H\leq N | H \\\ is \\\ a \\\ 2-absorbing \\\ pure\\\  submodule\\\ of \\\ M \}.
$$
Then $0 \in \Sigma$ implies that $\Sigma \not = \emptyset$. Let $\{N_\lambda\}_{\lambda \in \Lambda}$ be a
totally ordered subset of $\Sigma$. Then  $\cup_{\lambda \in \Lambda}N_\lambda\leq N$ and by Proposition \ref{p5.8} (a),  $\cup_{\lambda \in \Lambda}N_\lambda$ is a 2-pure submodule of $M$. Thus by using Zorn's Lemma, one can see that
$\Sigma$ has a maximal element, $K$ say as needed.
\end{proof}

\begin{defn}\label{d7.212}
We say that a pure submodule $N$ of an $R$-module $M$
is a  \emph {maximal pure submodule} of a submodule
$K$ of $M$, if $N \subseteq K$ and there does not exist a pure submodule $H$ of $M$ such that $N \subset H \subset K$.
\end{defn}

An $R$-module $M$ is called a \emph{fully cancellation module} if for each non-zero ideal $I$ of $R$ and for each submodules $N_1, N_2$ of $M$ such that $IN_1=IN_2$
implies $N_1=N_2$ \cite{HE14}.

\begin{thm}\label{c3.10}
Every Artinian fully cancellation $R$-module $M$ has only a finite
number of maximal pure submodules.
\end{thm}
\begin{proof}
Suppose that the result is false. Let $\Sigma$ denote the
collection of non-zero submodules $N$ of $M$ such that $N$ has an
infinite number of maximal pure submodules. The collection
$\Sigma$ is non-empty because $M \in \Sigma$ and hence has a minimal
member, $S$ say. Then $S$ is not a pure submodule. Thus there exists an ideal $I$ of $R$ such that $IS \not= S \cap IM$. Let $V$ be a maximal pure submodule of $M$ contained in $S$. If $IS+V=S$, then $(IS+V)\cap IM=S\cap IM$. Hence by the modular law, $IS+V \cap IM=S\cap IM$. Now as  $V$ is a pure submodule of $M$, $IS+IV=S \cap IM$. It follows that $S \cap IM \subseteq IS$, a contradiction. If $IS+V=V$, then $IS \subseteq V$ and so $IS=IS \cap IM\subseteq  V \cap IM=IV$. Thus $IV=IS$. Since $M$ is a fully cancellation $R$-module, $V=S$, a contradiction. Therefore, $V \subset IS+V \subset S$. Now by the choice of $S$, the module $IS+V$ has only finitely many maximal pure submodules. Therefore, there is only a finite number of possibilities for the module $S$ which is a contradiction.
\end{proof}

\begin{defn}\label{d3.1}
 We say that an $R$-module $M$ is \emph{fully 2-pure} if every submodule of $M$ is 2-pure.
\end{defn}

Let $N$ and $K$ be two submodules of $M$. The \emph{product} of $N$ and $K$ is defined by $(N:_RM)(K:_RM)M$  and denoted by $NK$ \cite{AF007}.
\begin{thm}\label{t3.11}
Let $M$ be a multiplication $R$-module. Then the following statements are equivalent.
\begin{itemize}
  \item [(a)] For submodules $N, K, H$ of $M$, we have $NHK=NK \cap NH \cap KH$.
  \item [(b)] $M$ is a fully 2-pure $R$-module.
\end{itemize}
\end{thm}
\begin{proof}
$(a)\Rightarrow (b)$. Let $N$ be a submodule of $M$ and $I$, $J$ be two proper ideals of $R$. Then by part (a) and the fact that $M$ is a multiplication $R$-module,
$$
IN \cap JN \cap IJM=(IM)(N) \cap (JM)(N) \cap (IJM)(M)=(IN)(JN)(IJM)\subseteq
$$
$$
(N^2)(IJM)\subseteq N(IJM)\subseteq IJN.
$$
The reverse inclusion is clear.

$(b)\Rightarrow (a)$. We have
$$
NK \cap NH \cap KH=(N:_RM)(K:_RM)M \cap (N:_RM)H \cap (K:_RM)H=
$$
$$
(N:_RM)(K:_RM)H=NKH.
$$
\end{proof}

\begin{thm}\label{t3.12}
Let $M$ be a finitely generated faithful multiplication $R$-module and $N$ be a submodule of $M$. Then $N$ is a  2-pure submodule of $M$ if and only if $(N:_RM)$ is a  2-pure ideal of $R$.
\end{thm}
\begin{proof}
Since $M$ is a multiplication $R$-module, $N=(N:_RM)M$.
Let $N$ be a 2-pure submodule of $M$ and let $I$ and $J$ be any two
proper ideals of $R$. Then $IJN = IN \cap JN \cap IJM$. Hence,
$$
 IJ(N:_RM)M = I(N:_RM)M \cap J(N:_RM)M \cap IJM.
$$
This implies that $IJ(N:_RM)M \supseteq (I(N:_RM) \cap J(N:_RM) \cap IJ)M$. Now by \cite[3.1]{BS89},  $IJ(N:_RM) \supseteq I(N:_RM) \cap J(N:_RM) \cap IJ$. This implies that $(N:_RM)$ is a  2-pure ideal of $R$ since the reverse inclusion is clear. Conversely, let $(N:_RM)$ be a  2-pure ideal of $R$ and let $I$ and $J$ be any two
proper ideals of $R$. Then
 $IJ(N:_RM) = I(N:_RM) \cap J(N:_RM) \cap IJ$.  Hence $IJ(N:_RM)M = (I(N:_RM) \cap J(N:_RM) \cap IJ)M$. Thus
$IJ(N:_RM)M = I(N:_RM)M \cap J(N:_RM)M \cap IJM$ by  \cite[3.1]{BS89}. Therefore, $IJN = IN \cap JN \cap IJM$ as desired.
\end{proof}

\section{$n$-pure submodules}
\begin{defn}\label{d15.1}
Let $n$ be a positive integer. We say that a submodule $N$ of an $R$-module $M$ is a \emph {$n$-pure submodule of $M$} if $I_1I_2...I_nN=I_1N \cap I_2N\cap...I_nN\cap (I_1I_2...I_n)M$ for all proper ideals $I_1, I_2,...I_n$ of $R$. Also, we say that an ideal $I$ of $R$ is a \emph {$n$-pure ideal of $R$} if $I$ is a $n$-pure submodule of $R$.
\end{defn}

\begin{rem}\label{r15.2}
Let $n$ be a positive integer.
Clearly every $(n-1)$-pure submodule of an $R$-module $M$ is a $n$-pure submodule of $M$. But we see in the Example \ref{e15.3} that the converse is not true in general.
\end{rem}

\begin{ex}\label{e15.3}
Let $n$ be a positive integer.
The submodule $\bar{2}\Bbb Z_{2^n}$ of the $\Bbb Z_{2^n}$-module $\Bbb Z_{2^n}$ is a $n$-pure submodule but it is not a  $(n-1)$-pure submodule.
\end{ex}

\begin{ex}\label{e515.3}
Let $n$ be a positive integer.
\begin{itemize}
  \item [(a)] The submodule $2\Bbb Z$ of the $\Bbb Z$-module $\Bbb Z$ is a not a $n$-pure submodule.
  \item [(b)] Let $n>1$ be an integer.  Since
  $$
  2^n\Bbb Z=\underbrace{(2\Bbb Z)...(2\Bbb Z)}_\text{$n$ times} \Bbb Z\not =\underbrace{ (2\Bbb Z)\Bbb Z \cap (2\Bbb Z) \Bbb Z \cap...\cap (2\Bbb Z)\Bbb Z }_\text{$n$ times}\cap (2^n\Bbb Z) \Bbb Q=2\Bbb Z,
$$
 the submodule $\Bbb Z$ of the $\Bbb Z$-module $\Bbb Q$ is not $n$-pure.
\end{itemize}
\end{ex}

\begin{prop}\label{l115.3}
Let $N$ be a submodule of an $R$-module $M$ and $n$ be a positive integer. If $I_1...I_nN=I_1N \cap I_2N...\cap I_nN$ for all proper ideals $I_1, I_2,...,I_n$ of $R$, then $N$ is a $n$-pure submodule of $M$.
  \end{prop}
\begin{proof}
Use the technique of  Proposition \ref{l15.3} (a).
\end{proof}

\begin{prop}\label{p15.9}
Let $M$ be an $R$-module,  $N\subseteq K$ be two submodules of $M$, and $n$ be a positive integer.  Then the $n$-absorbing purity relation satisfies the following.
\begin{itemize}
  \item [(a)] Transitive: If $N$ is a $n$-pure submodule of $K$ and $K$ is a $n$-pure submodule of $M$, then $N$ is a $n$-pure submodule of $M$.
  \item [(b)] Hereditary: If  $K$ is a $n$-pure submodule of $M$, then $K/N$ is a $n$-pure submodule of $M/N$.
  \item [(c)] If $K/N$ is a $n$-pure submodule of $M/N$ and $N$ is a pure submodule of $M$, then $K$ is a $n$-pure submodule of $M$.
\end{itemize}
\end{prop}

\begin{proof}
Use the technique of Proposition \ref{p5.9}.
\end{proof}

\begin{prop}\label{p15.6}
Let $n$ be a positive integer.
A submodule $N $ of an $R$-module $M $ is a $n$-pure submodule if and only if  $N_P$ is a  $n$-pure submodule of $M_P$ for every maximal ideal $P$ of $R$.
\end{prop}
\begin{proof}
Use the technique of Proposition \ref{p5.6}.
\end{proof}

\begin{prop}\label{p15.8}
Let $M$ be an $R$-module and $n$ be a positive integer. Then we have the following.
\begin{itemize}
  \item [(a)] If $\{N_\lambda\}_{\lambda \in \Lambda}$ is a chain of $n$-pure submodules of $M$, then
  $\cup_{\lambda \in \Lambda}N_\lambda$ is  a $n$-pure submodules of $M$.
  \item [(b)] If $\{N_\lambda\}_{\lambda \in \Lambda}$ is a chain of submodules of $M$ and $K$ is a $n$-pure submodule of $N_\lambda$ for each $\lambda \in \Lambda$, then $K$ is a $n$-pure submodule of $\cup_{\lambda \in \Lambda}N_\lambda$.
\end{itemize}
\end{prop}
\begin{proof}
Use the technique of Proposition \ref{p5.8}.
\end{proof}

\begin{thm}\label{p15.9}
Let $N$ be a submodule of an $R$-module $M$ and let $n$ be a positive integer.  Then there is a submodule $K$ of $N$ maximal with respect to $K \subseteq N$ and $K$ is a $n$-pure submodule of $M$.
\end{thm}
\begin{proof}
Use the technique of Theorem \ref{p5.9}.
\end{proof}

\begin{defn}\label{d13.1}
Let $n$ be a positive integer.
 We say that an $R$-module $M$ is \emph{fully $n$-pure} if every submodule of $M$ is $n$-pure.
\end{defn}

Let $N$ and $K$ be two submodules of $M$. The \emph{product} of $N$ and $K$ is defined by $(N:_RM)(K:_RM)M$  and denoted by $NK$ \cite{AF007}.
\begin{thm}\label{t13.11}
Let $M$ be a multiplication $R$-module and let  $n$ be a positive integer.  Then the following statements are equivalent.
\begin{itemize}
  \item [(a)] For submodules $N_1, N_2,...,N_n$ of $M$, we have $N_1N_2...N_n=N_1N_2 \cap N_1N_3 \cap ...\cap N_1N_n \cap (N_2N_3...N_n)$.
  \item [(b)] $M$ is a fully $n$-pure $R$-module.
\end{itemize}
\end{thm}
\begin{proof}
Use the technique of Theorem \ref{t3.11}.
\end{proof}

\begin{thm}\label{t13.12}
Let $M$ be a finitely generated faithful multiplication $R$-module,  $N$ be a submodule of $M$, and let $n$ be a positive integer. Then $N$ is a  $n$-pure submodule of $M$ if and only if $(N:_RM)$ is a  $n$-pure ideal of $R$.
\end{thm}
\begin{proof}
Use the technique of  Theorem \ref{t3.12}.
\end{proof}

\textbf{Acknowledgments.} The author would like to thank Prof. Habibollah Ansari-Toroghy for his helpful suggestions and useful comments.

\bibliographystyle{amsplain}

\end{document}